\documentclass[11pt,oneside]{amsart}
\usepackage{graphicx} 

\usepackage{amsfonts,latexsym,rawfonts,amsmath,amssymb,amsthm, mathrsfs, lscape}
\usepackage{fullpage}
\usepackage{graphicx}
\usepackage[all]{xy}
\usepackage[english]{babel}
\usepackage[utf8]{inputenc}
\usepackage{xfrac}
\usepackage[shortlabels]{enumitem}
\usepackage{comment}
\usepackage{cite}

\usepackage{array, tabularx}

\usepackage{bm}
\bmdefine{\bX}{X}
\bmdefine{\ba}{$\alpha$}

\usepackage{setspace}
\usepackage{textcomp}

\usepackage{color}

\usepackage{xparse}
\usepackage{soul}

\usepackage{graphicx}

\usepackage{latexsym}
\usepackage{tikz-cd}
\usepackage{geometry}
\usepackage[english]{babel}
\usepackage[T1]{fontenc}
\usepackage[utf8]{inputenc}
\usepackage{graphicx}
\usepackage{graphics}
\usepackage{mathrsfs}
\usepackage{amssymb}
\usepackage{amstext}
\usepackage{times}
\usepackage{cancel}
\usepackage{epsfig}
\usepackage{yfonts}
\usepackage{amsthm}
\usepackage{amsfonts}
\usepackage{amsmath}
\theoremstyle{plain}
\newtheorem{teorema}{Theorem}[section]
\newtheorem*{teorema*}{Theorem}
\newtheorem{lemma}[teorema]{Lemma}

\newtheorem{proposizione}[teorema]{Proposition}
\theoremstyle{definition}

\theoremstyle{remark}
\newtheorem{remark}[teorema]{Remark}

\def \p {\partial}
\def \bpa {\overline{\partial}}

\allowdisplaybreaks

\usepackage[hypertexnames=false,
backref=page,
    pdfpagemode=UseNone,
    breaklinks=true,
    extension=pdf,
    colorlinks=true,
    linkcolor=blue,
    citecolor=blue,
    urlcolor=blue,
]{hyperref}

\title[Chern-Ricci flat balanced metrics on small resolutions of Calabi-Yau threefolds]{Chern-Ricci flat balanced metrics on small resolutions of Calabi-Yau threefolds}

\author{Federico Giusti}
\address{Instituto de Ciencias Matemáticas, Calle Nicolás Cabrera 13-15, 28049 Madrid, Spain}
\email{federico.giusti@icmat.es}
\author{Cristiano Spotti}
\address{Department of Mathematics, Aarhus University, Ny Munkegade 118, 8000 Aarhus C, Denmark}
\email{c.spotti@math.au.dk}
\date{\today}

\begin{document}

\begin{abstract} 
Given a nodal Kähler Calabi-Yau threefold, we show, via a gluing construction, that all its - possibly non-Kähler - small resolutions admit Chern-Ricci flat balanced metrics, which among other things solve the conformally balanced equation appearing in the Hull-Strominger system.
\end{abstract}

\maketitle 

\section*{Introduction}

In the study of Calabi-Yau threefolds, ordinary double points have been playing a central role in the attempt of understanding the moduli space of these manifolds. This connection is mainly due to Reid's geometrization conjecture known as \textit{Reid's Fantasy} (see \cite{R}), which states that all compact Calabi-Yau threefolds can be connected through a finite number of \textit{conifold transitions} (introduced by Clemens and Friedman, see \cite{F}), i.e. a procedure consisting of the contraction of a finite family of disjoint $(-1,-1)$-curves in a compact Calabi-Yau threefold, followed by the smoothing of the ordinary double points obtained from the previous step. However, the manifolds resulting from the process do not necessarily carry Kähler metrics, making the study of this singular transitions a problem to be treated, from a metric point of view, in the non-Kähler setting. Indeed, Fu, Li and Yau (in \cite{FLY}) showed that smoothings \emph{sufficiently close} to the singular threefold always carry balanced metrics - i.e. Hermitian metrics with coclosed fundamental form - whose behaviour in a neighborhood of the vanishing cycles is exactly (a rescaling of) the Candelas-de la Ossa-Stenzel metric (see \cite{CO} and \cite{St}). The importance of this construction lies in a conjectural picture from Yau (see \cite{Y1}), which identifies the Hull-Strominger system (introduced in \cite{Hu} and \cite{S} in the framework of heterotic string theory) as a fundamental tool in geometrization problems in non-Kähler Calabi-Yau geometry (i.e. complex manifolds with holomorphically trivial canonical bundle), indicating its solutions as a possible instrument for a metric approach to Reid's fantasy (in this direction, progress was made by Collins, Picard and Yau in \cite{CPY1}). The Hull-Strominger system has proven to be an extremely hard problem to face (Yau has conjectured an algebro-geometric characterization of the existence of solutions in \cite{Y}), and up to now, only few solutions were found (see for example \cite{AGF}, \cite{CPY2}, \cite{FeY}, \cite{FuY} and \cite{LY}).\par

The following note focuses on the opposite direction of Fu, Li and Yau's construction, addressing the problem of producing - via gluing construction (following classical constructions in Kähler geometry as \cite{AP} and \cite{BM}) - Chern-Ricci flat balanced metrics on compact small resolutions of singular Kähler Calabi-Yau threefolds whose singular set consists  of a finite family of ordinary double points. This result is meant to solve the issue described in Section 4 of the recent paper \cite{GS} by the authors. The statement is the following.
\begin{teorema*}\label{main}
Let $(\tilde{M},\tilde{\omega})$ be a Kähler Calabi-Yau nodal threefold (with $\tilde{\omega}$ a singular Calabi-Yau metric), and let $M$ be a compact (not necessarily Kähler) small resolution of $\tilde{M}$. Then for sufficiently small $\varepsilon$, $M$ admits a Chern-Ricci flat balanced metric $\hat{\omega}$ such that $[\hat{\omega}^2]=[p^*\tilde{\omega}^2]$, where $p:M \rightarrow \tilde{M}$ is the resolution map, and $[\omega^2]$ converges to a nef class.
\end{teorema*}
The main difficulty, as discussed in \cite{GS}, is given by the fact that the Candelas-de la Ossa Calabi-Yau metric on the small resolution of the conifold has quadratic decay to the cone, which combined with the absence of an explicit form of the cut-off metric did not allow us to get clearer information on the behaviour of the volume. However, with more careful computations, we were able to see that, despite the slow decay in the cut-off region, the Chern-Ricci potential has a fourth order decay, guaranteed by the (surprising) fact that the cut-off procedure preserves the primitivity of the slow-decay piece of the metric.\par
The outline of the paper is the following. In Section 1 we recall the construction of the pregluing metric, slightly expanding the brief explaination from \cite[Section 4]{GS}. In Section 2 we show how to obtain a refined estimate on the behaviour of the volume of the cut-off asymptotically exactly conical balanced metric obtained on the small resolution of the conifold in the previous section, and quickly walk through the main analytic steps to achieve the wanted Chern-Ricci flat balanced metric.

\par 
\vspace{0.5cm}
{\bf Acknowledgements.} Both the authors are thankful to Villum Fonden for the funding through the Villum Young Investigator 0019098 grant, which in particular funded the first author's PhD position, during which part of this work was carried out.
\par 
The first named author would like to thank Institute of Advanced Study in Mathematics at Zhejiang University, along with Song Sun, for hosting him during his visit during which this paper was partially completed. He also acknowledges support by GNSAGA of INdAM.\par
Both the authors are grateful to Hans-Joachim Hein and Yifan Chen for the useful conversations that lead to the completion of this work.

\numberwithin{teorema}{section}

\section{The pre-gluing metric}\label{2}
In this section we will focus on recalling the construction of the pre-gluing metric, as done in \cite{GS} in the following result.
\begin{proposizione}[Proposition 1.2 \cite{GS}]
    Let $(\tilde{M},\tilde{\omega})$ be a smoothable projective Kähler Calabi-Yau nodal threefold (with $\tilde{\omega}$ a singular Calabi-Yau metric), and let $M$ be a compact (not necessarily Kähler) small resolution of $\tilde{M}$. Then $M$ admits a balanced approximately Chern-Ricci flat metric $\omega$ in the class $[p^*\tilde{\omega}^2]$, where $p$ is the resolution map.
\end{proposizione}
Before describing the steps to prove the above Proposition, we shall briefly recall the geometry of the \emph{conifold}, i.e. the singularity model of ordinary double points on threefolds, in order to introduce some crucial ingredients for our construction.

\subsection{The geometry of the conifold}
Ordinary double points on threefolds have been vastly studied in literature, hence we can gather here many information extremely useful to describe the singularity model. First of all, we shall recall that ordinary double points are modelled on the singular algebraic variety $X$, given by
\[
X:=\{ z_1^2+z_2^2+z_3^2+z_4^2=0\} \subseteq \mathbb{C}^4,
\]
known as the standard \textit{conifold}, whose only singular point is the origin (whose name is now clear from the equation describing it). While this description appears to be the easiest, it is not the most convenient when it comes to the study of small resolutions of these singularities. Hence we shall focus on an alternate one, identifying the conifold with the total space of a rank $2$ vector bundle over $\mathbb{P}^1$, with its zero-section collapsed to a point. Indeed, following for example \cite{CPY1}, it is easily seen that
\[
X\simeq \mathcal{O}_{\mathbb{P}^1}(-1)^{\oplus 2} \setminus \mathbb{P}^1,
\]
which immediately presents the total space $\hat{X}$ of the bundle $\mathcal{O}_{\mathbb{P}^1}(-1)^{\oplus 2}$ as the model for small resolutions of ordinary double points, identifying $\mathbb{P}^1$ as the exceptional set, which we shall refer to as the \emph{exceptional curve}.
\begin{remark}
    The reason why we always use the plural for small resolutions is because the conifold admits two bimeromorphic non-biholomorphic small resolutions, related by a transformation known as \emph{Atiyah-flop} (see \cite{A}, but also \cite{CO} and \cite{TY}). This in particular implies that a singular threefold with $m$ ordinary double points will admit $2^m$ holomorphically distinct small resolutions.
\end{remark}
If we now call $\pi:\hat{X}\rightarrow \mathbb{P}^1$ the bundle projection, $\omega_{FS}$ the Fubini-Study metric on $\mathbb{P}^1$, and $h_{FS}$ the line-bundle metric naturally induced on $\mathcal{O}_{\mathbb{P}^1}(-1)$ through the identification
\[
\mathcal{O}_{\mathbb{P}^1}(-2)\simeq K_{\mathbb{P}^1},
\]
we can define locally the function
\begin{equation}\label{defr}
    r(z,u,v)^3:=|u|_{h_{FS}}^2+|v|_{h_{FS}}^2=h_{FS}(z)(|u|^2+|v|^2)=(1+|z|^2)(|u|^2+|v|^2),
\end{equation}
where $(u,v)$ are fibre coordinates and $z$ is a coordinate on $\mathbb{P}^1$. This definition actually extends globally to the whole $\mathcal{O}_{\mathbb{P}^1}(-1)^{\oplus 2}$ minus the zero section, and showcases $r$ as the conical distance from the singularity of $X$. In particular, it allows us to produce the conical Calabi-Yau metric 
\begin{equation}
    \omega_{co,0}:=\frac{3}{2}i\p\bpa r^2,
\end{equation}
which we shall refer to as the \emph{standard cone metric}. This metric will play a crucial role in our construction, as it will act as a \emph{bridge} in the same way the flat metric does in the orbifold setting faced in \cite{GS} as well as in many Kähler scenarios such as in \cite{AP} and \cite{Sz}.
\begin{remark}
    A further description of the conifold can be given as the total space (minus the zero section) of the line bundle $$ \mathcal{O}_{\mathbb{P}^1 \times \mathbb{P}^1} (-1) \rightarrow \mathbb{P}^1 \times \mathbb{P}^1,$$
     The relevance of this identification lies in the fact that it presents the standard cone metric $\omega_{co,0}$ as the corresponding cone metric of the Sasaki-Einstein structure on the link $L\simeq S^3\times S^2$.
\end{remark}
With all this in mind, we shall now start walking through the construction of the pregluing metric, where at each step we shall recall the further results and ingredients we will need.\par
The setting is the following: let $\tilde{M}$ be a projective Kähler Calabi-Yau nodal threefold endowed with $\tilde{\omega}$ a singular Calabi-Yau metric and let $M$ be a compact (not necessarily Kähler) small resolution. For simplicity, in what follows, we will assume to be working with just one singularity, but everything works in the same way if we consider a finite family of ordinary double points. Also, thanks to the identification of a neighborhood of the ODP with the a nighborhood of the origin in the singularity model $X$, we can easily extend the function $r$ introduced above on $X$ to a smooth function on the whole $\tilde M$, which we will still indicate with $r$.
\subsection{Step 1}
The first step consists in modifying the metric in a neighborhood of the singularity in order to make it exactly conical (i.e. equal to $\omega_{co,0}$) around the ODP, while maintaining the metric to be "good", which in our setting means balanced and \emph{approximately} Chern-Ricci flat. This can be achieved without much difficulty thanks to the following result from Hein and Sun (see \cite{HS}) and Zhang \cite{Z}. 

\begin{teorema}[\cite{HS, Z}]\label{hs}
Let $\tilde{M}$ be a singular threefold whose singular set is a finite family of ODPs endowed with a Kähler Calabi-Yau metric $\tilde{\omega}$ on its smooth part $M_{reg}$. Then for every singular point $x \in \tilde{M}\setminus M_{reg}$ there exist a constant $\lambda_0 >0$, neighborhoods $x \in U_x \subseteq \tilde{M}$ and $0 \in V_x \subseteq X$, and a biholomorphism $P: V_x \setminus \{0\} \rightarrow U_x \setminus \{x \}$ such that
\[
P^*\tilde{\omega}-\omega_{co,0}=i\partial\overline{\partial} \varphi, \qquad \text{for some} \> \varphi \in C_{2+\lambda_0}^{\infty},
\]
where $r$ is the conical distance from the singularities and $C_{2+\lambda_0}^{\infty}$ is the space of smooth functions with decay rate at zero of $2+\lambda_0$ (i.e. an $f \in C_{2+\lambda_0}^{\infty}$ is a smooth function such that nearby zero it holds $|\nabla^k f| \leq cr^{2+\lambda_0-k}$ for all $k \geq 0$).
\end{teorema}
\begin{remark}
    The above Theorem as originally proved by Hein and Sun, requires the threefold to be smoothable. However, the hypothesis of smoothability can actually be lifted, thanks to a recent result of Zhang (see \cite{Z}, Theorem $1.4$) which extends the result also to the non-smoothable (local) case. 
\end{remark}
Thus, assuming that the neighborhood $U_x$ in Theorem \ref{hs} contains $\{r\leq 1\}$, for every $\varepsilon>0$ we can consider a smooth cut-off function $\chi_\varepsilon(r):=\chi(\varepsilon^{-p} r)$, with
\[
    \chi(y):=
        \begin{cases}
        0 \qquad & \text{on} \> [0,1], \\
        \text{non decreasing} \qquad & \text{on} \> [1,2], \\
        1 \qquad & \text{on} \> [2,+\infty].
        \end{cases}
\]
With this we can easily define $\tilde{\omega}_\varepsilon:=\omega_{co,0}+i\p\bpa(\chi_\varepsilon(r)\varphi)$ on $\{r\leq 1\}$, which clearly extends to the whole $\tilde M$ as $\tilde{\omega}$, providing a close $(1,1)$-form on the whole $\tilde M$ which is positive outside the cut-off region. It is however straightforward to notice that on $C_\varepsilon:=\{\varepsilon^p<r<2\varepsilon^p\}\subseteq \tilde{M}$ it holds
\[
\tilde{\omega}_\varepsilon=\omega_{co,0}+O(r^{\lambda_0}),
\]
guaranteeing the positivity also in the cut-off region $C_\varepsilon$ up to choosing sufficiently small $\varepsilon$, showing that $\tilde\omega_\varepsilon$ is a Kähler (hence also balanced) metric on $\tilde M$ coinciding with $\omega_{co,0}$ around the singularity. Morever, it is clear that $\tilde\omega_\varepsilon$ is Calabi-Yau outside the cut-off region, while on $\{\varepsilon^p<r<2\varepsilon^p\}$ it is easily noticed that $\det (\tilde\omega_\varepsilon)=1+O(r^{\lambda_0})$ showing that the Chern-Ricci potential of $\tilde\omega_\varepsilon$ decays as $O(r^{\lambda_0})$, making the metric \emph{approximately} Chern-Ricci flat.

\subsection{Step 2}\label{step2}
In this second step we will work instead on the singularity small resolution model $\hat{X}$, with the aim of producing a metric suitable to be glued to the one constructed in Step 1, in order to produce a global balanced \emph{approximately} Chern-Ricci flat metric on the small resolution $M$. The main ingredient here is the asymptotically conical Calabi-Yau metric $\omega_{co,1}$ constructed by Candelas and de la Ossa in \cite{CO}, which takes the form
\[
\omega_{co,1}=i\p\bpa f_1(r^3)+4\pi^*\omega_{FS}.
\]
Its property of being asymptotically conical is highlighted by the following expansion for the potential $f_1$, given in \cite{CPY1}.
\begin{lemma}\label{expansion}
For $y\gg 1$, the function $f_1(y)$ has a convergent expansion
\[
f_1(y)=\frac{3}{2}y^{\frac{2}{3}}-2\log(y)+\sum_{n=0}^{+\infty}c_n y^{-\frac{2n}{3}}.
\]
\end{lemma}
In particular this Lemma implies that, away from the exceptional curve (i.e. for large values of $r$, here intended as the conical distance on the whole Calabi-Yau cone $X$) we can write
\begin{equation}\label{espco}
\omega_{co,1}=\omega_{co,0}-6i\p\bpa\log r +4\pi^*\omega_{FS}+i\p\bpa(O(r^{-2})),
\end{equation}
showing \emph{en passant} that asymptotically it holds
\begin{equation}
    \omega_{co,1}=\omega_{co,0}+O(r^{-2}).
\end{equation}
\begin{remark}
    While the fact that the small resolution $M$ needs not to be Kähler can be shown algebraically (see for example \cite[Remark 4.5]{GS}), the expansion in equality \eqref{espco} already is a symptom of this, as it shows that $\omega_{co,1}-\omega_{co,0}$ is \emph{not} asymptotically exact, not allowing us to perform a Kähler cut-off to the cone metric $\omega_{co,0}$.
\end{remark}
Following the above Remark, since we still want to perform a cut-off to the cone metric $\omega_{co,0}$, we consider instead the square of $\omega_{co,1}$, which can be written as
\begin{equation}
    \omega_{co,1}^2=\omega_{co,0}^2+2i\omega_{co,0}\wedge(4\pi^*\omega_{FS}-6i\p\bpa\log r)+i\p\bpa(O(r^{-2}\log r)),
\end{equation}
which has now $i\p\bpa$-exact difference with the square of the cone metric.\par
We can then introduce, for all $R>0$ sufficiently large, a cut-off function $\xi_R(r^2):=\xi(R^{-2}r^2)$, where
\[
    \xi(y):=
        \begin{cases}
        1 \qquad & \text{on} \> [0,1/16], \\
        \text{non increasing} \qquad & \text{on} \> [1/16,1/4], \\
        0 \qquad & \text{on} \> [1/4,+\infty],
        \end{cases}
\]
and define on $\{r\geq R/4\}$
\begin{equation}\label{cutco}
    \omega_R^2:=\omega_{co,0}^2+2i\p\bpa(3/2r^2\xi_R(r^2))\wedge(4\pi^*\omega_{FS}-6i\p\bpa\log r)+i\p\bpa(\xi_R(r^2)O(r^{-2}\log r)),
\end{equation}
which again naturally extends to the whole $\hat{X}$ as $\omega_{co,1}$.\par
Exactly as above, we have produced a closed form which is also positive outside of the region $C_R:=\{R/4<r<R/2\}\subseteq \hat{X}$ since it corresponds to the square of a Kähler metric. Regarding the cut-off region, from expression \eqref{cutco} it is not hard to notice that it holds
\begin{equation}\label{roughest}
\omega_R^2=\omega_{co,0}^2+O(R^{-2}),
\end{equation}
ensuring the positivity of $\omega_R^2$, implying in turn that writing it as the square of a real $(1,1)$-form makes sense (thanks to the work of Michelsohn in \cite{M}) and that $\omega_R$ defines indeed an exactly asymptotically conical balanced metric on $\hat{X}$. On top of this, Kähler Ricci-flatness holds outside of $C_R$, while on $C_R$ we can combine identity \eqref{roughest} and Michelsohn's result (as in \cite[Remark 2.8]{GS}) to see that it holds $\det(\omega_R)=1+O(r^{-2})$, giving us that $\omega_R$ is also approximately Chern-Ricci flat.
\begin{remark}
    This step is slightly different from how it was carried out in \cite{GS}. Said difference consists in the placement of the cut-off function in the $i\p\bpa$-potential: the choice made here will be more suitable for the computations that will take place in Section \ref{sec2}.
\end{remark}
\subsection{Step 3}
Finally, exactly as in \cite{GS}, by referring to the "small" coordinates around $x$ in $\tilde M$ with $w$ and to the "big" coordinates far from the exceptional curve in $\hat{X}$, we can take the biholomorphism between cones identifying $\tilde{M}\supseteq\{\varepsilon^p/2<r(w)<\varepsilon\}\simeq \{R/2<r(\zeta)<R\}\subseteq \hat{X}$, exactly as in \cite{GS} yields the rescaling factor $\lambda=\frac{\varepsilon^{2p}}{R^2}$ which allows us to patch together $\lambda\omega_R$ with $\tilde{\omega}_\varepsilon$ on the identified region, producing a global balanced metric $\omega=\omega_{\varepsilon,R}$ on the whole $M$ described as
\[
\omega=
\begin{cases} \lambda\omega_{R} \qquad & \text{on} \> r(\zeta)\leq R/2, \\
\omega_{co,0} \qquad & \text{on}\> \varepsilon^p/2 \leq r(w) \leq \varepsilon^p, \\ 
\omega_\varepsilon \qquad & \text{on} \> r(w) \geq 2\varepsilon^p,
\end{cases}
\]
which we might informally refer to as the \emph{connected sum} $\lambda\omega_R\#\tilde\omega_\varepsilon$.\par
At this point, following again \cite{GS}, we unify the parameters by imposing $R:=\varepsilon^{-q}$, with $q>0$, and obtain $\lambda=\varepsilon^{2(p+q)}$, together with the estimate
\begin{equation}
    \omega=\omega_{co,0}+O(r^m)
\end{equation}
on the whole gluing region $\{\varepsilon^p/4<r(w)<\varepsilon^2\}$, where $m=\min\{\lambda_0,2\frac{q}{p}\}$. As in \cite{GS}, we can also deduce an estimate for the behaviour of the Chern-Ricci potential $$f:=\log\left(\frac{i\Omega\wedge\overline{\Omega}}{\omega^3}\right)\underset{loc.}{=}\det(\omega)=O(r^m)$$ (where $\Omega$ is the holomorphic volume for $M$ described as in \cite[Subsection 2.5]{GS}), i.e. $\omega$ is a balanced approximately Chern-Ricci flat metric on $M$.\par
However, as discussed in \cite[Section 4]{GS}, this behaviour is not good enough in order to be able to deform (with our choice of an ansatz) the metric $\omega$ to a genuine Chern-Ricci flat balanced metric in the balanced class $[\omega^2]=[p^*\tilde\omega^2]$, hence in the next section we will first show how in reality, the metric obtained by cutting-off as we did is already behaving correctly for us to conclude.

\section{Improving the behaviour of the Chern-Ricci potential and deforming to a genuine Chern-Ricci flat balanced metric}\label{sec2}
In this final Section we will first show how to improve the estimate on the Chern-Ricci potential, followed by a brief discussion on how now the deformation argument allows us to actually obtain a Chern-Ricci flat balanced metric on $M$.

\subsection{Improving the estimate on the Chern-Ricci potential}\label{impr}
Let us go back to the notation of Subsection \ref{step2}, hence we initally forget that we chose $R=\varepsilon^{-q}$, and (just to lighten the notation) go back to work with $\omega_R$ on $\hat{X}$. We then fix a point $x_o \in \hat{X}$ in the cut-off region $C_R$, and choose coordinates $(z,u,v)$ (defined on an open set containing $x_o$) such that $z(x_o)=0$. Our goal is to write in this coordinates the slow-decay part of $\omega_R^2$, in order to understand better how it appears in $\omega_R$ (which up to now we are not able to write it explicitly) and ultimately how it contributes to the volume.\par
First of all the choice of the coordinate $z$ gives us that
\begin{equation}\label{derivateh}
    h_{FS}(x_o)=1, \qquad \p h_{FS}(x_o)=0,\qquad i\p\bpa h_{FS}(x_o)=idz\wedge d\bar{z},
\end{equation}
implying also that 
$$r_o^3:=r^3(x_o)=|u_o|^2+|v_o|^2,$$
with $(u_o,v_o)$ the fibre coordinates corresponding to $x_o$. With this in mind, we can obtain the following expressions at $x_o$:
\begin{equation}\label{conocoord}
\begin{aligned}
    (\omega_{co,0})_{|_{x_o}}=\left(\frac{3}{2}i\p\bpa r^2\right)_{|_{x_o}}= & ir_o^{-1}\left( \left(1-\frac{r_o^{-3}}{3}|u_o|^2\right)du\wedge d\bar u + \left(1-\frac{r_o^{-3}}{3}|v_o|^2\right)dv\wedge d\bar v \right.\\
    & \left. -\frac{r_o^{-3}}{3}(\bar u_o v_o du\wedge d\bar v +u_o \bar v_0 dv\wedge d\bar u)\right)+r_o^2idz\wedge d\bar z;
\end{aligned}
\end{equation}
\begin{equation}\label{logcoord}
    \begin{aligned}
        (6i\p\bpa\log r)_{|_{x_o}}= & 2ir_o^{-3}((1-r_o^{-3}|u_o|^2)du\wedge d\bar u + (1-r_o^{-3}|v_o|^2)dv\wedge d\bar v  \\ & -r_o^{-3}(\bar u_o v_odu\wedge d\bar v + u_o\bar v_o dv\wedge d\bar u))+2idz\wedge d\bar z;        
    \end{aligned}
\end{equation}
\begin{equation}
    \begin{aligned}
        (4\pi^*\omega_{FS})_{|_{x_o}}=4(i\p\bpa \log h_{FS})_{|_{x_o}}=4\left(\frac{1}{h_{FS}^2}i\p h_{FS}\wedge\bpa h_{FS}+\frac{1}{h_{FS}}i\p\bpa h_{FS}\right)_{|_{x_o}}=4idz\wedge d\bar z.
    \end{aligned}
\end{equation}
We now want to introduce a change of coordinates on the fibres in order to make the cone metric $\omega_{co,0}$ into the identity at $x_o$. As we will see, this choice will end up making diagonal at the same time also the other significant pieces in the description of $\omega_R^2$. Consider then the new (complex) coordinates $(x,y)$ on the fibres, related to $(u,v)$ by
\[
    u=u_o \sqrt{3/2}r_o^{-1}x-r_o^{-1}\bar v_o y, \qquad
    v=\sqrt{3/2}r_o^{-1}v_o x+r_o^{-1}\bar u_o y,
\]
along with a rescaling of the coordinate $z$ by a factor $r_o^{-1}$, which we shall do without renaming of the coordinate. With these new coordinates $(x,y,z)$, we obtain 
\begin{equation}\label{conoid}
    (\omega_{co,0})_{|_{x_o}}=idx\wedge d\bar x+idy\wedge d\bar y+ idz\wedge d\bar z;
\end{equation}
\begin{equation}\label{logdiag}
    (6i\p\bpa\log r)_{|_{x_o}}=2r_o^{-2}(idy\wedge d \bar y+idz\wedge d\bar z);
\end{equation}
\begin{equation}\label{fsrisc}
    (4\pi^*\omega_{FS})_{|_{x_o}}=4r_o^{-2}idz\wedge d\bar z,
\end{equation}
which are now suitable for our purpose. In the first place we directly get the following remark.
\begin{remark}
    Identities \eqref{conoid}, \eqref{logdiag} and \eqref{fsrisc} allow us to immediately notice that the slow decay term 
    $$4\pi^*\omega_{FS}-6i\p\bpa\log r$$
    is $\omega_{co,0}$-primitive. This primitivity can however be obtained without computations applying the work of Goto (see \cite{Go}).
\end{remark}
Moreover, if we expand
\begin{equation}\label{espcut}
    \begin{aligned}
        \frac{3}{2}i\p\bpa(r^2\xi_R(r^2))= & \left(2\frac{r^2}{R^2}\xi'(r^2/R^2)+\frac{r^4}{R^4}\xi''(r^2/R^2)\right)6i\p r \wedge\bpa r\\
        & + \left(\xi(r^2/R^2)+\frac{r^2}{R^2}\xi'(r^2/R^2) \right)\omega_{co,0},
    \end{aligned}
\end{equation}
and recall (say from \cite{CPY1}) that
\begin{equation}\label{idcpy}
    6i\p r \wedge\bpa r=\omega_{co,0}-3r^2i\p\bpa\log r,
\end{equation}
combining identities \eqref{conoid}, \eqref{logdiag}, \eqref{espcut} and \eqref{idcpy} yields
\begin{equation}\label{expcutoff}
        \left(\frac{3}{2}i\p\bpa(r^2\xi_R(r^2))\right)_{|_{x_o}}=  \frac{2}{3}(\alpha_R+\beta_R)idx\wedge d\bar x-\frac{2}{3}(\beta_R-5\alpha_R)(idy\wedge d\bar y - idz\wedge d\bar z),
\end{equation}
where $\alpha_R=\alpha_R(r)=2\frac{r^2}{R^2}\xi'(r^2/R^2)+\frac{r^4}{R^4}\xi''(r^2/R^2)$ and $\beta_R=\beta_R(r)=\xi(r^2/R^2)+\frac{r^2}{R^2}\xi'(r^2/R^2)$, which are both $O(1)$ uniformly in $R$.\par
Going now back to expression \eqref{cutco}, and plugging in \eqref{conoid}, \eqref{logdiag}, \eqref{espcut} and \eqref{expcutoff} at the point $x_o$, gives
\begin{equation}\label{espcoordcutoff}
    \begin{aligned}
        (\omega_R^2)_{|_{x_o}}= & -2(dx\wedge d\bar x \wedge dy\wedge d\bar y+dx\wedge d\bar x \wedge dz\wedge d\bar z+dy\wedge d\bar y \wedge dz\wedge d\bar z)\\
        & + (\alpha_R+\beta_R)idx\wedge d\bar x\wedge2r_o^{-2}(idz\wedge d\bar z-idy \wedge d\bar y)\\
        &-(\beta_R-5\alpha_R)(idy\wedge d\bar y - idz\wedge d\bar z)\wedge2r_o^{-2}(idz\wedge d\bar z-idy \wedge d\bar y)+O(r_o^{-4})\\
        = & -2(1-r_o^{-2}(\alpha_R+\beta_R))dx\wedge d\bar x \wedge dy\wedge d\bar y\\
        & -2(1+r_o^{-2}(\alpha_R+\beta_R))dx\wedge d\bar x \wedge dz\wedge d\bar z-2(dy\wedge d\bar y \wedge dz\wedge d\bar z)+O(r_o^{-4})\\
 = & (1-r_o^{-4}\gamma_R^2)^2 ( idx\wedge d\bar x+(1+r_o^{-2}\gamma_R)^{-2} idy\wedge d\bar y \\ &+(1-r_o^{-2}\gamma_R)^{-2}idz\wedge d\bar z )^2+O(r_o^{-4}),
    \end{aligned}
\end{equation}
where $$\gamma_R=\gamma_R(r)=\xi(r^2/R^2)+3\frac{r^2}{R^2}\xi'(r^2/R^2)+\frac{r^4}{R^4}\xi''(r^2/R^2)=O(1) \quad \text{uniformly in $R$}.$$
Applying now Michelsohn's result and arguing again as in \cite[Remark 2.8]{GS}, we obtain 
\begin{equation}\label{cutcord}
    (\omega_R)_{|_{x_o}}=(1-r_o^{-4}\gamma^2)\left(idx\wedge d\bar x+(1+r_o^{-2}\gamma_R)^{-2} idy\wedge d\bar y+(1-r_o^{-2}\gamma_R)^{-2}idz\wedge d\bar z\right)+O(r_o^{-4}),
\end{equation}
from which it is easily seen to hold
\[
(\omega_R)_{|_{x_o}}^3=(\omega_{co,0}^3)_{|_{x_o}}+O(r_o^{-4}),
\]
i.e. the faster decay of the volume that we wanted. On top of this, from identities \eqref{conoid}, \eqref{logdiag} and \ref{fsrisc} we get 
\begin{equation}
\begin{cases}
    (idz\wedge d\bar z)_{|_{x_o}}=(r^2\pi^*\omega_{FS})_{|_{x_o}};\\
    (idy\wedge d \bar y)_{|_{x_o}}=(3r^2i\p\bpa\log r-r^2\pi^*\omega_{FS})_{|_{x_o}};\\
    (idx\wedge d\bar x)_{|_{x_o}}=(\omega_{co,0}-3r^2i\p\bpa\log r)_{|_{x_o}};
\end{cases}    
\end{equation}
which combined with equation \eqref{cutcord} allows us to obtain an asymptotic expression on the whole cut-off region $C_R$ for $\omega_R$, that is
\begin{equation}\label{globcutoff}
    \omega_R=\omega_{co,0}+\gamma_R(r)(4\pi^*\omega_{FS}-6i\p\bpa\log r)+O(r^{-4}).
\end{equation}
This last expression shows how the the slow decay part keeps being primitive (with respect to the cone metric), as it appears to only be multiplied by a scalar function, ensuring that it will not give \emph{harmful} contributions to the volume behaviour, and guarantee that on the whole $C_R$ it holds
\begin{equation}\label{newvoldecay}
    \omega_R^3=\omega_{co,0}^3+O(r^{-4}).
\end{equation}
Hence, if we go back to the setting obtained at the end of Section \ref{2} we actually have that
\begin{equation}
    f=O(r^{\tilde m}),
\end{equation}
where this time $\tilde m:=\min\{\lambda_0, 4\frac{q}{p}\}$, which is now more then sufficient to allow us to complete the deformation argument as we will briefly explain in the next subsection.
\subsection{Deforming to a genuine solution}
We can now go back to the deformation argument as described in \cite[Section 4]{GS}, and consider once again the Fu-Wang-Wu deformation (see \cite{FWW}) with the ansatz considered in \cite{GS}, i.e.
\begin{equation}
    \omega_\psi^2=\omega^2+i\p\bpa(\psi\omega), \qquad \text{for}\>\> \psi \in C^2, \> \omega_\psi^2>0,
\end{equation}
which allows us to reformulate the \emph{balanced Monge-Ampére} equation
\begin{equation}\label{balma}
    \omega_\psi^3=e^f\omega^3.
\end{equation}
Here, following \cite{Sz}, we consider a variation of the equation by introducing an evaluation of the input function at a point $s \in \tilde M$ away from the gluing region, by substituting $f$ with $f-ev_s(\psi)$ in \eqref{balma}, and rephrasing the new equation as
\begin{equation}\label{balopmod}
\tilde{F}(\psi):=\frac{\omega_\psi^n}{\omega^n}-e^{f-ev_s{\psi}}=0.
\end{equation}
We then proceed in the exact same way, and first of all obtain the linearized operator at $\psi=0$, that is
\begin{equation}
    \tilde{L}u=\Delta_\omega u+\frac{1}{2}|\p\omega|_\omega^2 u+u(s)e^f,
\end{equation}
for which we are able to reprove uniform invertibility. This time, thanks to the Kählerianity of $\tilde\omega$, we can obtain it without any further assumption (as compared to \cite[Lemma 3.4]{GS}). As in our previous work, we will be working in weighted Hölder spaces as considered in \cite[Section 3.2]{GS}.
\begin{lemma}\label{invlin}
For every $b \in (0,2)$ it exists a uniform $c>0$ such that for sufficiently small $\varepsilon$ the operator $\tilde{L}$ is (uniformly) invertible and for all $u \in C_{\varepsilon,b}^{2,\alpha}$ it holds
\[
||u||_{C_{\varepsilon,b}^{2,\alpha}} \leq c||\tilde{L}u||_{C_{\varepsilon,b+2}^{0,\alpha}}.
\]
\end{lemma}
\begin{proof}
    The proof is actually a simplified version of the one of \cite[Lemma 3.4]{GS}, with the difference that the role that was of the flat metric on $\mathbb{C}^3\setminus\{0\}$ is now of the conifold $X$ endowed with the cone metric $\omega_{co,0}$. We will thus only show that there is no kernel involved on the "$\tilde M$ side". Indeed on $\tilde M$ the kernel of the "limit operator" ends up being described by
    \[
    \Delta_{\tilde{\omega}}u+u(s)=0.
    \]
    Multiplying this last expression with the volume form $\tilde\omega^3/3!$ and integrating everything on $\tilde M$ allows us to conclude (through integration by parts, which is possible thanks to the assumption $b<2$) that $u (s)=0$. Thus $u$ is harmonic on $\tilde M$, and at the singularity blows up slower then the laplacian's Green's function (again thanks to $b<2$), forcing $u$ to be constant. But now $u(s)=0$ imposes that the constant must be zero, resulting in a vanishing kernel.
\end{proof}
Once this is established, we see that searching solutions of equation \eqref{balopmod} (which clearly still produces Chern-Ricci flat balanced metrics) can be turned into a fixed point problem, and this can be achieved by considering the auxiliary operators $\hat{F}, E, G: C_{\varepsilon,b}^{2,\alpha}(M) \rightarrow C_{\varepsilon,b+2}^{0,\alpha}(M)$ defined as
\[
\hat{F}(\psi):=\frac{\omega_\psi^3}{\omega^3}, \quad E(\psi):=e^{f-ev_s(\psi)} \quad \text{and} \quad G(\psi)=e^fev_s(\psi).
\]
Expanding then
\[
\hat{F}(\psi)=\hat{F}(0)+L(\psi)+\hat{Q}(\psi),
\]
we can rephrase $\tilde{F}(0)=0$ as
\[
\hat{F}(0)+\tilde{L}(\psi)+\hat{Q}(\psi)+G(\psi)-E(\psi)=0, 
\]
which thanks to Lemma \ref{invlin} becames the fixed point problem
\begin{equation}
\psi=\tilde{L}^{-1}(E(\psi)-G(\psi)-\hat{F}(0)-\hat{Q}(\psi))=:N(\psi),
\end{equation}
with $N: C_{\varepsilon,b}^{2,\alpha}(M) \rightarrow C_{\varepsilon,b}^{2,\alpha}(M)$.\par

At this stage, we identify the (same) open set $U_\tau$, with $\tau>0$, given by
\[
U_\tau:=\{\varphi \in C_{\varepsilon,b}^{2,\alpha} \>|\> ||\varphi||_{C_{\varepsilon,b}^{2,\alpha}} < \tilde{c}\varepsilon^{(p+q)(b+2)+\tau} \} \subseteq C_{\varepsilon,b}^{2,\alpha},
\]
over which we get (again in a simplified way as in \cite{GS}) that $N$ is a contraction operator. On top of this, what we established in Subsection \ref{impr} allows us to actually see that
\[
N(U_\tau) \subseteq U_\tau,
\]
since it allows us to choose $p,q>0$ such that $p\tilde m-q(b+2)>\tau>0$, which for all $\psi \in U_\tau$ gives
\[
\begin{aligned}
||N(\psi)|||_{C_{\varepsilon,b}^{2,\alpha}} \leq & ||N(\psi)-N(0)|||_{C_{\varepsilon,b}^{2,\alpha}}+||N(0)|||_{C_{\varepsilon,b}^{2,\alpha}} \\
\leq & c\varepsilon^\tau ||\psi|||_{C_{\varepsilon,b}^{2,\alpha}}+||\tilde{L}^{-1}(1-e^f)|||_{C_{\varepsilon,b}^{2,\alpha}} 
\leq  c\varepsilon^\tau ||\psi|||_{C_{\varepsilon,b}^{2,\alpha}}+||f||_{C_{\varepsilon,b+2}^{0,\alpha}} \\
\leq & c(\varepsilon^{(p+q)(b+2)+2\tau}+\varepsilon^{p(b+2)+p\tilde m}) 
\leq  \tilde{c}\varepsilon^{(p+q)(b+2)+\tau}.
\end{aligned}
\]
Hence we are able to produce a Chern-Ricci flat balanced metric $\hat{\omega}$ on the small resolution $M$.
\begin{remark}
    Such a metric corresponds to a solution of the \emph{conformally balanced equation} from the Hull-Strominger system. It is thus a natural problem to try to understand, in the same fashion as Yau's proposed approach to Reid's Fantasy, if such small resolution could always carry solutions tu the Hull-Strominger system.
\end{remark}

\end{document}